\RequirePackage{fix-cm}
\documentclass[smallextended]{svjour3}       
\smartqed  
\usepackage{graphicx}
\usepackage{latexsym,amsmath,amssymb,graphicx}
\usepackage[all,web]{xy}
\usepackage{hyperref}

\def\urlfont{\DeclareFontFamily{OT1}{cmtt}{\hyphenchar\font='057}
             \normalfont\ttfamily \hyphenpenalty=10000}

\DeclareFontFamily{OT1}{rsfs10}{}
\DeclareFontShape{OT1}{rsfs10}{m}{n}{ <-> rsfs10 }{}
\DeclareMathAlphabet{\mathscript}{OT1}{rsfs10}{m}{n}

\DeclareMathOperator{\im}{Im}       
\DeclareMathOperator{\Spec}{Spec}   
\DeclareMathOperator{\Ext}{Ext}     
\DeclareMathOperator{\Exc}{Exc}     
\DeclareMathOperator{\ext}{\mathcal{E}\it{x}\hskip1pt \it{t}\hskip1pt} 
\DeclareMathOperator{\hm}{\mathcal{H}\it{o}\hskip1pt \it{m}\hskip1pt} 
\DeclareMathOperator{\Sing}{Sing}   
\DeclareMathOperator{\Def}{Def}     
\DeclareMathOperator{\dimdef}{def}     
\DeclareMathOperator{\codim}{codim} 

\def\P{{\mathbb{P}}}

\def\p2{\mathbb{P}^2}
\def\p3{\mathbb{P}^3}
\def\p4{\mathbb{P}^4}

\def\su{\operatorname{SU}}

\def\GL{\operatorname{GL}}

\def\C{\mathbb{C}}

\def\T{\mathbb{T}}
\def\gt{T(\widehat{X},X,\widetilde{X})}

\newcommand{\oneline}{\vskip12pt}
\newcommand{\halfline}{\vskip6pt}
\newcommand{\cy}{Ca\-la\-bi--Yau }
\newcommand{\ka}{K\"{a}hler }

\begin{document}

\title{A small and non-simple geometric transition
\thanks{The author was partially supported by MIUR-PRIN 2010-11 Research Funds ``Geometria delle Variet\`{a} Algebriche'' and by the I.N.D.A.M. as a member of the G.N.S.A.G.A.}
}

\titlerunning{A small and non-simple geometric transition}        

\author{Michele Rossi}

\authorrunning{Michele Rossi} 

\institute{M. Rossi \at
              Dipartimento di Matematica, Universit\`a di Torino,
via Carlo Alberto 10, 10123 Torino \\
              Tel.: +39-011-6702813\\
              Fax: +39-011-6702878\\
              \email{michele.rossi@unito.it}          }

\date{}

\maketitle

\begin{abstract}
Following notation introduced in the recent paper \cite{Rdef}, this paper is aimed to present in detail an example of a \emph{small} geometric transition which is not a \emph{simple} one i.e. a \emph{deformation} of a conifold transition. This is realized by means of a detailed analysis of the Kuranishi space of a Namikawa cuspidal fiber product, which in particular improves the conclusion of Y.~Namikawa in Remark 2.8 and Example 1.11 of \cite{N}. The physical interest of this example is presenting a geometric transition which can't be immediately explained as a massive black
hole condensation to a massless one, as described by A.~Strominger \cite{Strominger95}.
\end{abstract}

\tableofcontents

\section{Introduction}
\label{intro}
Let $X$ be a complex projective threefold with terminal
singularities and admitting a \emph{small} resolution
$\widehat{X}\stackrel{\phi}{\longrightarrow}X$ such that
$\widehat{X}$ is a \emph{\cy threefold} (in the sense of Definition \ref{cy-def}), where ``small" means that
the exceptional locus $\Exc (\phi)$ has codimension greater than or
equal to two. Then it is well known that $\Exc (\phi)$ consists of
a finite disjoint union of trees of rational curves of A--D--E
type \cite{Reid83}, \cite{Laufer81}, \cite{Pinkham81},
\cite{Morrison85}, \cite{Friedman86}. In his paper \cite{N},
Remark 2.8, Y.~Namikawa considered the following

\halfline
\noindent\textbf{Problem} \emph{When does $\widehat{X}$ have a flat deformation
such that each tree of rational curves splits up into mutually
disjoint \emph{$(-1,-1)$--curves}?}

Let us recall that a \emph{$(-1,-1)$--curve} is a rational curve
in $X$ whose normal bundle is isomorphic to
$\mathcal{O}_{\P^1}(-1)\oplus\mathcal{O}_{\P^1}(-1)$. It arises
precisely as exceptional locus of the resolution of an ordinary
double point (a \emph{node}) also called a \emph{conifold} point
since it is an isolated hypersurface
singularity whose tangent cone is a non--degenerate quadratic
cone.

Namikawa's problem is interesting in the
context of H.~Clemens type pro\-blems of cycle deformations (see
e.g. \cite{Friedman86}, Corollary (4.11)). Moreover, it is of
significant interest in the context of (deformations of) geometric transitions
and therefore in the study of the moduli space for \cy threefolds. Let us
recall that a \emph{geometric transition} (gt) between two \cy
threefolds is the process obtained by ``composing" a birational
contraction to a normal threefold with a complex smoothing (see
Definition \ref{gtdef}). If the normal intermediate threefold
has only nodal singularities then the considered gt is called a
\emph{conifold transition}. The interest in geometric transitions goes back to the
ideas of H.~Clemens \cite{Clemens} and M.~Reid \cite{Reid87} which
gave rise to the so called \emph{\cy Web Conjecture} (see also
\cite{Gross97b} for a revised and more recent version) stating that
(more or less) all \cy threefolds can be connected to each other by
means of a chain of geometric transitions, giving a sort of (unexpected)
``connectedness" of the moduli space for \cy threefolds. There is also
a considerable physical interest in geometric transitions owing to the fact that
they connect topologically distinct models of \cy vacua: the
physical version of the \cy Web Conjecture is a sort of (in this
case expected) ``uniqueness" of a space--time model for
supersymmetric string theories (see e.g. \cite{CGH90} and
references therein).

In this context, Namikawa's problem can then be rephrased as
follows

\halfline
\noindent\textbf{Problem}(for small geometric transitions) \emph{When does a small gt
have the same ``deformation type" (see Definition \ref{def:def}) of a conifold transition?}

Since the geometry of a general gt can be very intricate, while
the geometry of a conifold transition is relatively easy and well
understood as a topological surgery \cite{Clemens}, the
mathematical interest of such a problem is evident.

On the other hand, conifold transitions were the first (and among
the few) geometric transitions to be physically understood as massive black
holes condensation to massless ones, by A.~Strominger \cite{Strominger95}. Answering
the given problem would then give a significant improvement in
the physical interpretation of (at least the small) geometric transitions
bridging topologically distinct \cy vacua.

\halfline Unfortunately in \cite{N}, Remark 2.8, Namikawa observed
that a flat deformation positively resolving the given problem
``\emph{does not hold in general}" and produced an example of a
\emph{cuspidal fiber self--product of an elliptic rational surface
with sections} whose resolution admits exceptional trees, composed
of couples of rational curves intersecting at one point, which
should not deform to a disjoint union of $(-1,-1)$--curves.
Nevertheless, Example 1.11 in \cite{N}, supporting such a
conclusion, did not give a correct argument since the proposed deformation is actually a trivial one.

In the present paper we will overcome this argument by giving a detailed analysis of the Kuranishi space of the Namikawa cuspidal fiber product, allowing us to go far beyond his conclusion in \cite{N},
Remark 2.8. Moreover this will give rise to an explicit example of a small gt which is not a simple one i.e. it has not the same deformation type of a conifold transition: such an example has already been sketched in \S 9.2 of \cite{Rdef}, without any proof. Here all the needed details will be given.

\halfline The paper is organized as follows.

In Section \ref{sec:1} we introduce notation, preliminaries and main
facts needed throughout the paper. Section \ref{sec:2} is then devoted to present the Namikawa construction of
a fiber self--product of a particular elliptic rational surface
with sections and singular ``cuspidal" fibers (which will be
called \emph{cuspidal fiber product}). These are threefolds
admitting six isolated singularities of Kodaira type $II\times II$
which have been rarely studied in either the pioneering work of
C.~Schoen \cite{S} or the recent \cite{Kap}. For this reason,
their properties, small resolutions and local deformations are
studied in detail. In particular, all the local deformations
induced by global versal deformations are studied in Proposition
\ref{def-singolari}, while all the local deformations of a
cuspidal singularity to three distinct nodes are studied in
Proposition \ref{3sings-effettive}. They actually do not lift globally to
the given small resolution, as stated by Theorem \ref{immagine0},
revising the Namikawa considerations of \cite{N}, Remark 2.8 and
Example 1.11: see Remark~\ref{rem:nodef} and Theorem~\ref{thm:nodef}. The last section \ref{s:gt} is dedicated to apply Theorem \ref{immagine0} to deformations of geometric transitions: for further details the interested reader is referred to \cite{Rdef} \S 7.1.

\section{Preliminaries and notation}
\label{sec:1}
\begin{definition}[\cy 3--folds]\label{cy-def}
A smooth, complex, projective 3--fold $X$ is called \emph{\cy} if
\begin{enumerate}
    \item $\mathcal{K}_X \cong \mathcal{O}_X$\ ,
    \item $h^{1,0}(X) = h^{2,0}(X) = 0$\ .
\end{enumerate}
\end{definition}

The standard example is the smooth quintic threefold in $\P^4$. The given definition is equivalent to require that $X$ has holonomy group a subgroup
    of $\su (3)$ (see \cite{Joice2000} for a complete description of
    equivalences and implications).

\subsection{Deformations of complex spaces}\label{deformazioniCY}
Let us start by recalling that a \emph{complex space} is a ringed space $(X,\mathcal{O}_X)$ where $X$ is a Hausdorff topological space locally isomorphic to a locally closed analytic subset of some $\C^N$ and $\mathcal{O}_X$ is the induced sheaf of holomorphic functions. A \emph{pointed complex space} is a pair $(X,x)$ consisting of a complex space $X$ and a distinguished point $x\in X$. A \emph{morphism $f:(X,x)\to(Y,y)$ of pointed complex spaces} is a morphism $f:X\to Y$ of complex spaces such that $f(x)=y$.

\emph{Complex space germs} are pointed complex spaces whose morphisms are given by equivalence classes of morphisms of pointed complex spaces defined in some open neighborhood of the distinguished point.  Let $(X,x)$ be a complex space germ and $U\subset X$ an open neighborhood of $x$: then the inclusion map $U\hookrightarrow X$ gives an isomorphism of complex space germs $(U,x)\cong(X,x)$. Then $U$ is called \emph{a representative of the germ $(X,x)$}.

\noindent  A complex space germ is also called a \emph{singularity}.

\begin{definition}[Deformation of a complex space, \cite{Palamodov} \S~5, \cite{ACG} \S XI.2] Let $X$ be a complex space. A \emph{deformation of $X$} is a flat, holomorphic map $f:\mathcal{X}\to (B,o)$  from the complex space $\mathcal{X}$ over a complex space germ $(B,o)$, endowed with an isomorphism $X\cong\mathcal{X}_o:=f^{-1}(o)$ on the central fibre, for short denoted by $f:(\mathcal{X}, X)\to (B,o)$. If $X$ is compact then we will also require that $f$ is a proper map. If $X$ is singular and the fibre $\mathcal{X}_b=f^{-1}(b)$ is smooth, for some
$b\in B$, then $\mathcal{X}$ is called \emph{a smoothing family of $X$}.
With a slight abuse of notation, if $b\neq o$ then the fiber $\mathcal{X}_b$ is called either a \emph{deformation} or a \emph{smoothing} of $X$ when $\mathcal{X}$ is a either a deformation or a smoothing family of $X$, respectively.
\end{definition}

\noindent Let $\Omega_{X}$ be the sheaf of holomorphic
differential forms on $X$ and consider the
\emph{Lichtenbaum--Schlessinger cotangent sheaves}
\cite{Lichtenbaum-Schlessinger} of $X$, $\Theta^i_{X} = \ext
^i\left(\Omega_{X},\mathcal{O}_{X}\right)$. Then $\Theta^0_X = \hm
\left(\Omega_X,\mathcal{O}_X\right)=: \Theta_X$ is the ``tangent"
sheaf of $X$ and $\Theta^i_X$ is supported over $\Sing(X)$, for
any $i>0$. Consider the associated \emph{local} and \emph{global
deformation objects}
\[
    T^i_X:= H^0(X,\Theta^i_X)\quad,\quad
    \T^i_X:=\Ext ^i \left(\Omega^1_X,\mathcal{O}_X\right)\ ,\
    i=0,1,2.
\]
Then by the \emph{local to global spectral
    sequence} relating the global $\Ext$ and sheaf $\ext$
    (see \cite{Grothendieck57} and \cite{Godement} II, 7.3.3) we get
    \begin{equation*}
    \xymatrix{E^{p,q}_2=H^p\left(X,\Theta_X^q\right)
               \ar@{=>}[r] & \mathbb{T}_X^{p+q}}
    \end{equation*}
giving that
\begin{eqnarray}
    &&\T^0_X \cong T^0_X \cong H^0(X,\Theta_X)\ , \label{T0} \\
    \label{T-lisci}&&\text{if $X$ is smooth then}\quad \T^i_X\cong H^i(X,\Theta_X)\ ,  \label{T1'}\\
    &&\text{if $X$ is Stein then}\quad T^i_X\cong\T^i_X\ .
    \label{Ti-Ti}
\end{eqnarray}
Given a deformation family
$\mathcal{X}\stackrel{f}{\longrightarrow}B$ of $X$ \emph{for each
point $b\in B$ there is a well defined linear (and functorial) map
}
\[
    D_b f: \xymatrix{T_b B\ar[r]& \T^1_{X_b}}\quad
    \text{(Generalized Kodaira--Spencer map)}
\]
where $T_b B$ denotes the Zariski tangent space to $B$ at $b$ (\cite{Greuel-etal} Lemma II.1.20, \cite{Palamodov2} \S~2.4).

For the following terminology the reader is referred to \cite{Looijenga} \S~6.C, \cite{Greuel-etal} \S~1.3,  \cite{Palamodov} Def.\ 5.1, \cite{Palamodov2} \S~2.6, among many others.
 A deformation $f:(\mathcal{X},X)\to (B,o)$ of a complex space $X$ is called \emph{versal} (some authors say \emph{complete}) if, for any deformation $g:(\mathcal{Y},X)\to
    (C,p)$ of $X$, there exists a holomorphic map of germs of complex spaces $h:(C,p)\to (B,o)$ such that $g=h^*(f)$ i.e. the following diagram commutes
    \[
        \xymatrix{\mathcal{Y}=U\times_B \mathcal{X}\ar[r]\ar[d]^-{g=h^*(f)}&\mathcal{X}\ar[d]^-f\\
                   C\ar[r]^-h&B}
    \]
    In particular the generalized Kodaira--Spencer map $\kappa(f)$ turns out to be surjective (\cite{Palamodov2}, \S~2.6). Moreover the deformation $f$ is said to be an \emph{effective versal} (or \emph{miniversal}) deformation of $X$
    it is versal and $\kappa(f)$ is injective, hence an isomorphism. Finally the deformation $f$ is called \emph{universal} if it is versal and $h$ is
    uniquely defined. This suffices to imply that $f$ is an effective versal deformation of $X$ (\cite{Palamodov2}, \S~2.7.1).

The following result is a central one in the theory of deformation of complex spaces: it is due to many authors as A.~Douady \cite{Douady74}, H.~Grauert \cite{Grauert72} and \cite{Grauert74}, V.P.~Palamodov  \cite{Palamodov72}, among many others.

\begin{theorem}[Existence of a versal deformation]\label{DGP teorema}
\begin{enumerate}
  \item Any representative $U$ of an isolated singularity $(X,x)$ has a miniversal deformation $$f:(\mathcal{U},U)\to(B,o)$$
  (see \cite{Grauert72}, \cite{Greuel-etal}).
  \item Any (compact) complex space $X$ has an effective versal deformation $$f:(\mathcal{X},X)\to (B,o)$$ (see \cite{Palamodov72} for the non compact case and \cite{Douady74}, \cite{Grauert74}, \cite{Palamodov} for the compact case).
  \item The germ of complex space $(B,o)$ obtained by the previous parts (1) and (2) is isomorphic to the germ  $(q_X^{-1}(0), 0)$, where $q_X:\T^1_X\rightarrow \T^2_X$ is a
suitable holomorphic map (the \emph{obstruction map}) such that
$q_X(0)=0$ (see \cite{Greuel-etal} II.1.5, \cite{Palamodov2} \S~2.5)
\end{enumerate}
\end{theorem}

In particular if $q_X\equiv 0$ (e.g. when $\T^2_X=0$) then $(B,o)$
turns out to be isomorphic to the complex space germ $(\T^1_X,0)$.

\begin{definition}[Kuranishi space and number]\label{Def-definizione}
The germ of complex space $(B,o)$ of parts (1) and (2) of the previous Theorem \ref{DGP teorema} is called the
\emph{Kuranishi space of either $(X,x)$ or $X$, respectively}: in the first case one can easily check that it does not depend on the choice of the representative $U$. The Kuranishi space is said to give an \emph{analytic representative} of the deformation functors $\Def(U)$ and $\Def(X)$, respectively, allowing one to set the identifications
\[
    \text{either}\ (\Def(X,x):=)\ \Def(U)=(B,o)\ \text{or}\ \Def(X)=(B,o)\ ,\ \text{respectively.}
\]
The \emph{Kuranishi numbers $\dimdef
(X,x)$ of $(X,x)$ and $\dimdef
(X)$ of $X$}, are then the maximum dimensions of irreducible
components of $\Def(X,x)$ and $\Def(X)$, respectively.

\noindent $\Def(X,x)$ and $\Def(X)$ are said to be
\emph{smooth}, or \emph{unobstructed}, if the obstruction map $q_X$ is the constant map
$q_X\equiv 0$. This means that any first order deformation arises to give a deformation of either $U$ or $X$, respectively. In this case
either $\Def(X,x)\cong(T^1_U,0)$ or $\Def(X)\cong(T^1_X,0)$, respectively,
giving either $\dimdef(X,x)=\dim_{\C}\T^1_U$ or $\dimdef(X)=\dim_{\C}\T^1_X$, respectively. By a slight abuse of notation, we will write
\[
    \text{either}\quad\Def(X,x)\cong\T^1_U\quad\text{or}\quad\Def(X)\cong\T^1_X\ ,\quad\text{respectively.}
\]
\end{definition}

\begin{theorem}[\cite{Palamodov} Thm. 5.5]\label{universality thm}
Let $X$ be a compact complex space such that $\T^0_X=0$. Then the versal effective deformation of $X$, given
in part (2) of Theorem \ref{DGP teorema}, is actually a universal deformation of $X$.
\end{theorem}

\begin{example}[The germ of complex space of an isolated hypersurface singularity] Let $f:\C^{n+1}\longrightarrow\C$, $n\geq 0$, be a holomorphic map
admitting an isolated critical point in the origin $0\in\C^{n+1}$ and consider the local ring $\mathcal{O}_0$ of germs of holomorphic
function of $\C^{n+1}$ in the origin. It is a well known fact that $\mathcal{O}_0$
is isomorphic to the ring of convergent power series
$\C\{x_1,\ldots,x_{n+1}\}$. A \emph{germ of hypersurface
singularity} $(U_0,0)$ is defined by means of the Stein complex space
\begin{equation}\label{germe}
    U_0:=\Spec(\mathcal{O}_{f,0})
\end{equation}
where $\mathcal{O}_{f,0}:=\mathcal{O}_0/(\overline{f})$ and $\overline{f}$ is the germ
represented by the series expansion of the above given holomorphic function $f$. Let $J_f\subset\C\{x_1,\ldots,x_{n+1}\}$ be the \emph{jacobian ideal} of $f$ i.e. the ideal generated by the partial derivatives of $f$. Then the following facts hold:
\begin{itemize}
  \item since $U_0$ is Stein, (\ref{Ti-Ti}) gives that $\T^i_{U_0}\cong T^i_{U_0}$\ ,
  \item since we are dealing with an isolated hypersurface singularity,
  \begin{equation}\label{ihs-T^i}
    \T^i_{U_0}\cong T^i_{U_0}=0\ ,\quad\text{for $i\geq2$}
  \end{equation}
  (for the case $i=2$, which is all what is useful in the following, see e.g. \cite{Stevens} \S~3 Example on pg.~26; for $i\geq 2$ see e.g. \cite{Greuel-etal} Prop.\ C.4.6(3)),
  \item the Kuranishi space of an isolated hypersurface singularity is then completely described as follows
\begin{equation}\label{ihs-T^1}
    \Def(U_0,0)\cong\T^1_{U_0}\cong \mathcal{O}_{f,0}/J_f\cong\left.\C\left\{x_1,\ldots,x_{n+1}\right\}\right/\left((\overline{f})+J_f\right)
\end{equation}
(the first isomorphism follows by Theorem \ref{DGP teorema} and the previous (\ref{ihs-T^i}); for the second isomorphism see e.g. \cite{Stevens} \S~3 Example on pg.~24, \cite{Greuel-etal} Corollary II.1.17).
\end{itemize}
\end{example}

\begin{example}[The deformation theory of a \cy threefold]
Let us now consider the case of a \emph{\cy threefold $X$}.  A central result in the deformation theory of \cy manifolds is the well--known Bogomolov--Tian--Todorov--Ran Theorem
\cite{Bogomolov78},\cite{Tian87},\cite{Todorov89},\cite{Ran92} asserting that the
Kuranishi space $\Def(X)$ \emph{is smooth}, hence $\Def(X)\cong\T^1_X$ and (\ref{T1'}) gives
that
\begin{equation}\label{kuranishi numero}
    \dimdef(X) = \dim_{\C}\T^1_X = h^1(X,\Theta_X) = h^{2,1}(X)
\end{equation}
where the last equality on the right is obtained by the \cy
condition $\mathcal{K}_X\cong\mathcal{O}_X$. Applying the \cy
condition once again gives $h^0(\Theta_X)=h^{2,0}(X)=0$. Therefore
(\ref{T0}) and Theorem \ref{universality thm} give the existence
of a \emph{universal effective family of \cy deformations of $X$}.
In particular $h^{2,1}(X)$\emph{ turns out to be the dimension of
the complex moduli space of $X$}.
\end{example}

\subsection{Deformation of a morphism}\label{ssez:morf-def}
Let us quickly recall the concept of \emph{deformation of a morphism} as defined by Z.~Ran in \cite{Ran89}.

 Consider a morphism $\phi:Y\rightarrow X$ of complex spaces and let $B$ be a connected complex space with a special point $o\in B$ such that $g:(\mathcal{Y},Y)\rightarrow(B,o)$ and $f:(\mathcal{X},X)\rightarrow(B,o)$ are deformation families of $Y$ and $X$, respectively. Then a \emph{deformation family of the morphism $\phi$} is a morphism $\Phi:\mathcal{Y}\rightarrow\mathcal{X}$ such that the following diagram commutes
\begin{equation}\label{deformazione di morfismo}
\xymatrix{Y\ar@{^{(}->}[rrr]\ar[rd]^{\phi}\ar[ddr]&&&\mathcal{Y}\ar[dl]_{\Phi}\ar[ddl]^{g}\\
          &X\ar@{^{(}->}[r]\ar[d]&\mathcal{X}\ar[d]_{f}&\\
          &o\ar@{}[r]|{\in}&B&}
\end{equation}
with $Y=g^{-1}(o)$, $X=f^{-1}(o)$
    and $\phi = \Phi_{|g^{-1}(o)}$.

    \noindent Given two distinct points $b_1,b_2\in B$ the morphism $\phi_2:=\Phi_{|g^{-1}(b_2)}$ is called a \emph{deformation} of the morphism $\phi_1:=\Phi_{|g^{-1}(b_1)}$ and viceversa.

\subsection{The Friedman diagram}

Let $\phi:Y\to X$ be a birational contraction of a \cy manifold $Y$, of dimension $n\geq 3$, to a normal variety $X$ with isolated rational singularities and assume that the codimension of the exceptional locus $E:=\Exc(\phi)$ is greater than or equal to 2: hence $\phi$ is what is usually called a \emph{small resolution} of $X$. This latter assumption ensures that the Friedman argument of  Lemma (3.1) in \cite{Friedman86} still applies to give that $R^0\phi_*\Theta_Y\cong\Theta_X$. Then we get the following commutative diagram, to which we refer as the \emph{Friedman diagram} (see \cite{Friedman86} (3.4)),
\begin{equation}\label{Friedman-diagramma}
    \xymatrix{H^1(R^0\phi_*\Theta_Y)\ar@{^{(}->}[r]\ar@{=}[d] &
    \mathbb{T}^1_Y\ar[r]^-{\lambda_E}\ar[d]^{\delta_1} &
    H^0\left(R^1\phi_*\Theta_Y\right)\ar[r]\ar[d]^{\delta_{loc}}& H^2(R^0\phi_*\Theta_Y)\ar[r]\ar@{=}[d]
     & \mathbb{T}^2_Y\ar[d]^-{\delta_2}\\
    H^1(\Theta_{X})\ar@{^{(}->}[r]&\mathbb{T}^1_{X}\ar[r]^-{\lambda_P}
     &T^1_{X}\ar[r]& H^2(\Theta_{X})\ar[r] & \mathbb{T}^2_{X}}
\end{equation}
where:
\begin{itemize}
    \item the first row is the lower terms exact sequence of the
    Leray spectral sequence of $\phi_*\Theta_Y$, where we can use
    (\ref{T1'}) as $Y$ is smooth;
    \item the second row is the lower terms exact sequence of the
    Local to Global spectral sequence converging to $\T^n_{X}=
    \Ext^n(\Omega^1_{X},\mathcal{O}_{X})$;
    \item the vertical maps
    \begin{equation*}
        \xymatrix{\T^i_Y=\Ext^i(\Omega_Y,\mathcal{O}_Y)\ar[r]^-{\delta_i}&\T^i_X=\Ext^i(\Omega_X,\mathcal{O}_X)}
    \end{equation*}
arise from the natural map $\phi^*\Omega_X\to\Omega_Y$ just recalling that $X$ has rational singularities, hence giving
$$\Ext^i(\Omega_X,\mathcal{O}_X)
=\Ext^i(\Omega_X,\phi_*\mathcal{O}_Y)=\Ext^i(\phi^*\Omega_X,\mathcal{O}_X)\ ;$$
\end{itemize}

\begin{proposition}[\cite{Friedman86} Prop.~(2.1)]\label{prop:localizzazione-diff}
Let $U_p$ be a Stein neighborhood of the singular point $p\in P=\Sing(X)$ and set $V_p:=\phi^{-1}(U_p)$. Let $\Def(U_p)$ and $\Def(V_p)$ be the Kuranishi spaces defined in parts 1 and 2 (the non-compact case) of Thm.~\ref{DGP teorema}, respectively. Then, under hypothesis given above, one gets:
\begin{enumerate}
    \item\ $\mathbb{T}^1_{V_p} \cong H^1\left(V_p,\Theta_{V_p}\right)\cong
H^0\left(R^1\phi_*\Theta_{V_p}\right)$ and
$H^0\left(R^1\phi_*\Theta_{Y}\right)\cong \bigoplus_{p\in P}
\T^1_{V_p}$ is the tangent space to $\prod_{p\in P}\Def(V_p)$,
    \item\ $\mathbb{T}^1_{U_p} \cong
    T^1_{U_p}$ and $T^1_{X}\cong
    \bigoplus_{p\in P} T^1_{U_p}$
    is the tangent space to $\prod_{p\in P}\Def(U_p)$,
    \item morphisms $\delta_{loc}$ and $\delta_1$ in the Friedman diagram
(\ref{Friedman-diagramma}) are injective.
\end{enumerate}
\end{proposition}

\begin{proof} Parts 1 and 2 follow immediately by Prop.~(2.1), part 1, in \cite{Friedman86}, recalling the previously displayed formulas (\ref{T1'}) and (\ref{Ti-Ti}). To prove part 3 notice, on the one hand, that the injectivity of $\delta_1$ follows from the injectivity $\delta_{loc}$, by an easy diagram chase. On the other hand, $\delta_{loc}$ is injective by Prop.~(2.1), part 2, in \cite{Friedman86}.
\end{proof}

\begin{remark} Since $Y$ is a \cy 3-fold, the Bogomolov--Tian--Todorov--Ran Theorem
gives that $\Def(Y)\cong H^1(\Theta_{Y})$. Moreover $\phi$ is a small resolution giving that $\Sing(X)$ is composed at most by
terminal singularities of index 1: in \cite{Namikawa94}, Theorem A, Y.~Namikawa proved an extension of
the Bo\-go\-mo\-lov--Tian--Todorov--Ran Theorem allowing to conclude
that $\Def(X)$ is smooth also in the present situation, hence giving $\Def(X)\cong\T^1_X$. By the previous Prop.~\ref{prop:localizzazione-diff}, the localization near to a singular point $p\in\Sing(X)$ of the second square on the left of the Friedman diagram (\ref{Friedman-diagramma}) can then be rewritten as follows
\begin{equation}\label{Friedman-locale}
    \xymatrix{\Def(Y)\cong H^1(\Theta_{Y})\ar[r]^-{\lambda_{E_p}}
    \ar@{^{(}->}[d]^-{\delta_1}&\Def(V_p)\cong H^0(R^1\phi_*\Theta_{V_p})\ar@{^{(}->}[d]^-{\delta_{loc,p}}\\
    \Def(X)\cong\T^1_X\ar[r]^-{\lambda_p}&\Def(U_p)\cong T^1_{U_p}}
\end{equation}
where $E_p=\phi^{-1}(p)$ is the exceptional locus over $p$. In the following we will refer to this diagram as the \emph{local Friedman diagram}. Notice that all the maps involved in this diagram are compatible with corresponding natural transformations between the associated deformation functors. For a deeper understanding of this fact the interested reader is referred to \cite[\S~1]{Wahl76}, \cite[(3.4)]{Friedman86}, \cite[(11.3-4)]{Kollar-Mori92} and references therein.
\end{remark}

\subsection{Fiber products of rational elliptic surfaces with sections}

In the present subsection we will review some well known facts about
rational elliptic surfaces with section and their fiber products.
For further details the reader is
referred to \cite{M}, \cite{MP} and \cite{S}.

Let $Y$ and $Y'$ be \emph{rational elliptic surfaces with
sections} i.e. rational surfaces admitting elliptic fibrations
over $\P^1$
\[
r\ :\ \xymatrix@1{Y\ar[r] &\ \P^1}\quad,\quad r'\ :\
\xymatrix@1{Y'\ar[r] &\ \P^1}
\]
with distinguished sections $\sigma_0$ and $\sigma'_0$,
respectively (notation as in \cite{S} and \cite{MP}). Define
\begin{equation}\label{fiber-prod}
    X:= Y\times_{\P^1}Y'\ .
\end{equation}
Write $S$ (resp. $S'$) for the images of the singular fibers of
$Y$ (resp. $Y'$) in $\P^1$.

\begin{proposition}\label{preliminari}\hfill
\begin{enumerate}
    \item The fiber product $X$ is smooth if and only if $S\cap S' =
    \emptyset$. In particular, if smooth, $X$ is a \cy threefold \emph{(\cite{S} \S2)} and $\chi(X)=0$.
    \item $Y$ (resp. $Y'$) is the blow--up of $\P^2$ at the base
    locus of a rational map $\varrho: \P^2\dashrightarrow\P^1$ (resp.
    $\varrho': \P^2\dashrightarrow\P^1$) \emph{(\cite{MP} Prop. 6.1)}.
    \item If $Y=Y'$ is \emph{sufficiently general} and $r(=r')$ admits at most nodal
fibers, then there always exists a small projective resolution
$\widehat{X}$ of $X$ \emph{(\cite{S} Lemma (3.1))}. Moreover if $Y$
     has exactly $\nu\geq 0$ nodal
    fibers then $\chi(X)=\nu$ and $\chi(\widehat{X})=2\nu$.
\end{enumerate}
\end{proposition}

\begin{theorem}[Weierstrass representation of an elliptic surface with
section, \cite{Kas} Thm. 1, \cite{M} Thm. 2.1, \cite{HL} \S
2.1 and proof of Prop. 2.1]\label{W-teorema} Let
$r:\widetilde{Y}\longrightarrow C$ be a relatively minimal elliptic surface
over a smooth base curve $C$, whose generic fibre is smooth and
admitting a section $\sigma:C\longrightarrow \widetilde{Y}$ (then $\widetilde{Y}$ is
algebraic \cite{K}). Let $\mathcal{L}$ be the co--normal sheaf of
$\sigma(C)\subset \widetilde{Y}$.

Then $\mathcal{L}$ is invertible and there exists
\[
A\in H^0(C,\mathcal{L}^{\otimes 4})\quad,\quad B\in
H^0(C,\mathcal{L}^{\otimes 6})
\]
such that $\widetilde{Y}$ is the minimal resolution of the closed subscheme $Y$ of the
projectivized bundle $\P(\mathcal{E}):=\P(\mathcal{L}^{\otimes
3}\oplus\mathcal{L}^{\otimes 2}\oplus\mathcal{O}_C)$ defined by
the zero locus of the homomorphism
\begin{eqnarray}\label{W-map}
    \quad\quad(A,B)&:&\xymatrix@1{\mathcal{E}=\mathcal{L}^{\otimes
    3}\oplus\mathcal{L}^{\otimes
    2}\oplus\mathcal{O}_C\ \ar[rrr]&&&\hskip1cm
    \mathcal{L}^{\otimes 6}}\\
\nonumber
    &&\xymatrix@1{\hskip1cm (x,y,z)\hskip1cm\ar@{|->}[rr]&& - x^2z + y^3 + A\ yz^2 + B\
    z^3}\ .
\end{eqnarray}
The pair $(A,B)$ (hence the homomorphism (\ref{W-map})) is
uniquely determined up to the transformation $(A,B)\mapsto(c^4 A,
c^6 B),\ c\in\C^*$ and the \emph{discriminant form}
\[
\delta:= 4\ A^3 + 27\  B^2 \in H^0(C,\mathcal{L}^{\otimes 12})
\]
vanishes at a point $\lambda\in C$ if and only if the fiber
$Y_{\lambda}:= r^{-1}(\lambda)$ is singular.
\end{theorem}

\begin{remark}\label{W-razionale}
Assume that the elliptic surface $r:Y\longrightarrow C$ is
\emph{rational}. Then \emph{$C\cong\P^1$ is a rational curve and
the section $\sigma(C)$ is a $(-1)$--curve in $Y$} (see \cite{M}
Proposition (2.3) and Corollary (2.4)). In particular
$\mathcal{L}\cong\mathcal{O}_{\P^1}(1)$ and (\ref{W-map}) is a
homomorphism
\begin{equation}\label{W-rat-map}
    \xymatrix@1{\mathcal{E}=\mathcal{O}_{\P^1}(3)\oplus\mathcal{O}_{\P^1}(2)
\oplus\mathcal{O}_{\P^1}\ar[rr]^-{(A,B)} && \
\mathcal{O}_{\P^1}(6)}\ .
\end{equation}
\end{remark}

\oneline Consider the fiber product
\[
X:= Y\times_{\P^1}Y
\]
of the Weierstrass fibration defined as the zero locus
$Y\subset\P(\mathcal{E})$ of the bundle homomorphism
(\ref{W-rat-map}). Hence, for generic $A,B$, the rational elliptic
surface $Y$ has smooth generic fiber and a finite number of
distinct singular fibers associated with the zeros of the
discriminant form $\delta = 4 A^3 + 27 B^2$. In general the
singular fibers are nodal and $\Sing(X)$ is composed by a finite
number $\nu = 12$ of distinct nodes. We can then apply
Proposition \ref{preliminari}(3) to guarantee the existence of a
small resolution $\widehat{X}\longrightarrow X$ whose exceptional
locus is the union of disjoint $(-1,-1)$--curves, i.e. rational
curves $C\cong\P^1$ in $X$ whose normal bundle is
$\mathcal{N}_{C|X}\cong
\mathcal{O}_{\P^1}(-1)\oplus\mathcal{O}_{\P^1}(-1)$.

\noindent Anyway, if either $A$ and $B$ have a common root or
$A\equiv 0$, the Weierstrass fibration $Y$ may admit
\emph{cuspidal fibers}: in this case the existence of a small
resolution for $X$ is no more guaranteed by Proposition
\ref{preliminari}(3).

\section{The Namikawa fiber product}
\label{sec:2}
In \cite{N}, \S 0.1,  Y. Namikawa considered the Weierstrass
fibration associated with the bundle homomorphism
\begin{eqnarray}\label{W-Nami}
    (0,B)&:&\xymatrix@1{\mathcal{E}=\mathcal{O}_{\P^1}(3)\oplus\mathcal{O}_{\P^1}(2)
\oplus\mathcal{O}_{\P^1}\ \ar[rr]&&\quad
    \mathcal{O}_{\P^1}(6)}\\
\nonumber
    &&\xymatrix@1{\hskip1cm (x,y,z)\hskip1cm\ar@{|->}[rr]&&\quad - x^2z + y^3 + B(\lambda)\
    z^3}
\end{eqnarray}
i.e. its zero locus $Y\subset\P(\mathcal{E})$. The associated
discriminant form is $\delta(\lambda)=27B(\lambda)^2\in
H^0(\mathcal{O}_{\P^1}(12))$ whose roots are precisely those of
$B\in H^0(\mathcal{O}_{\P^1}(6))$. Hence, for a generic $B$, the
rational elliptic surface $Y\longrightarrow\P^1$ has smooth
generic fiber and six distinct \emph{cuspidal fibers}.

\begin{proposition}\label{prop:cuspidi} The fiber product $X:= Y\times_{\P^1}Y$ is a threefold admitting 6
\emph{threefold cups} whose local type is described by the
singularity
\begin{equation}\label{cA2}
    X^2 - U^2 - Y^3 + V^3 \in \C[X,Y,U,V] \ .
\end{equation}
In the standard Kodaira notation these
are singularities of type $II\times II$ (\cite{K}, Theorem 6.2).
\end{proposition}

\begin{proof} Our hypothesis give
\begin{equation}\label{W-N}
\P(\mathcal{E})\times\P^1[\lambda] \supset Y\ :\  x^2z = y^3 +
B(\lambda)z^3\ .
\end{equation}
Then its fiber self--product can be represented as follows
\begin{equation}\label{Nami}
    \P:=\P(\mathcal{E})\times\P(\mathcal{E})\times\P^1[\lambda]\supset X\ :\
    \left\{\begin{array}{c}
    x^2z = y^3 + B(\lambda)z^3 \\
    u^2w = v^3 + B(\lambda)w^3 \\
    \end{array}\right.\ .
\end{equation}
Since the
problem is a local one, let us consider the open subset
$\mathcal{A}\subset\P$ defined
by setting
\begin{equation}\label{apertoW}
    \mathcal{A}:=\{(x:y:z)\times(u:v:w)\times(\lambda_0:\lambda_1)\ |\ z\cdot w\cdot \lambda_1\neq
    0\}\cong \C^5 (X,Y,U,V,t)
\end{equation}
where $X=x/z,Y=y/z,U=u/w,V=v/w$ and $t=\lambda_0/\lambda_1$. Then
$\mathcal{A}\cap X$ can be locally described by equations
\begin{equation}\label{eq.localeW}
\left\{\begin{array}{c}
    X^2 = Y^3 + B(t) \\
    U^2 = V^3 + B(t) \\
    \end{array}\right.\ ,
\end{equation}
where $B(\lambda)=\lambda_1^6
B(t)$. If $t_0$ is a zero of the discriminant $\delta(t)= 27 B(t)^2$ then
$p_{t_0}=((0:0:1),(0:0:1),(t_0:1))\in X$
is a singular point, whose local equations are obtained from
(\ref{eq.localeW}) by replacing $B(t)$ with its Taylor expansion in a neighborhood of $t_0$, giving
\begin{equation}\label{eq.localeW2}
\left\{\begin{array}{c}
    X^2 = Y^3 + t^iB^{(i)}(t_0)/i! +o(t^i) \\
    U^2 = V^3 + t^iB^{(i)}(t_0)/i! +o(t^i) \\
    \end{array}\right.\ ,
\end{equation}
where $i$ is the minimum order of derivatives such that $B^{(i)}(t_0)=dB^i(t_0)/dt^i$ does not vanish. Then the germ described by equations
(\ref{eq.localeW2}) turns out to be the same described by the local equation $X^2 - U^2 - Y^3 + V^3 = 0$.
\end{proof}

As already observed above, Proposition \ref{preliminari}(3) can
then no more be applied to guarantee the existence of a small
resolutions $\widehat{X}\longrightarrow X$. In this case Y.~Namikawa
proved the following

\begin{proposition}[\cite{N} Example in \S 0.1]\label{nami-risoluzione}
The \emph{cuspidal fiber product} $X= Y\times_{\P^1}Y$ associated
with the Weierstrass fibration $Y$, defined as the zero locus in
$\P(\mathcal{E})$ of the bundle homomorphism (\ref{W-Nami}),
admits six small resolutions which are connected to each other by
flops of $(-1,-1)$--curves. The exceptional locus of any such
resolution is given by six disjoint couples of $(-1,-1)$--curves
intersecting in one point.
\end{proposition}

\begin{proof} Let us first of all sketch the idea of the proof. The resolutions of $X$ are constructed by considering the strict transform of $X$ into a successive double blow up of the projectivized bundle $\P:=\P(\mathcal{E})\times\P(\mathcal{E})\times\P^1$ in which $X$ is embedded. Then the cusps of $X$ are firstly resolved to nodes and then finally resolved to smooth points. The key point of the construction is that $\Sing(X)$ is contained in the diagonal locus $\Delta$ of $\P$; moreover $X$ turns out to be invariant under the action of a cyclic group of order 6 acting on $\P$. Then the six resolutions of $X$ are constructed by a successive blow up of suitable couples of images of $\Delta$ under the action of this cyclic group.

 Give $X$ by equations (\ref{Nami}) and
consider the following cyclic map on $\P$
\begin{eqnarray}
   \tau\ :&\xymatrix@1{\P(\mathcal{E})\times\P(\mathcal{E})\times\P^1\quad\ar[rr]
           && \quad\P(\mathcal{E})\times\P(\mathcal{E})\times\P^1}&  \\
\nonumber
   &\xymatrix@1{((x:y:z),(u:v:w),\lambda)\ar@{|->}[r]
   & ((x:y:z),(-u:\epsilon v:w),\lambda)}&\ ,
\end{eqnarray}
where $\epsilon$ is a primitive cubic root of unity. The second
equation in (\ref{Nami}) ensures that $\tau X = X$. Since $\tau$
generates a cyclic group of order 6, the orbit of the codimension
2 diagonal locus
\[
\Delta:= \{((x:y:z),(u:v:w),\lambda)\in
\P(\mathcal{E})\times\P(\mathcal{E})\times\P^1\ |\ (x:y:z)=(u:v:w)\}
\]
is given by six distinct codimension 2 cycles $\{\tau^i\Delta\ |\
0\leq i\leq 5\}$. For any $i$, $\tau^i\Delta$ cuts on $X$ a Weil
divisor $D_i:=\tau^i\Delta\cap X$ containing $\Sing (X)$: in fact
\begin{eqnarray*}
    \Sing (X)&=&\left\{\left((0:0:1),(0:0:1),\lambda\right)\in\P\
    |\ B(\lambda)=0\right\}\subset  \\
    D_i&=&\left\{\left((x:y:z),((-1)^ix:\epsilon^i y:z),\lambda\right)\in\P\
    |\ x^2z - y^3 - B(\lambda)z^3=0\right\} \ .
\end{eqnarray*}
Then
$$\Sing(X)=\bigcap_{i=0}^5 D_i\ .$$
Let $\P_i$ be the blow--up of $\P$ along $\tau^i\Delta$: the
exceptional divisor is a $\P^1$--bundle over $\tau^i\Delta$. Let
$X_i$ be the strict transform of $X$ in the blow--up
$\P_i\longrightarrow\P$. Since $\Sing(X)$ is entirely composed by
singularities of type (\ref{cA2}),
$X_i$ is singular and $\Sing(X_i)$ contains only nodes. Moreover
$X_i\longrightarrow X$ turns out to be a small partial resolution
whose exceptional locus is a union of disjoint
$(-1,-1)$--curves, one over each singular point of $X$.

\noindent Consider the strict transform $(\tau^{i+1}\Delta)_i$ of
$\tau^{i+1}\Delta$ in the blow--up $\P_i\longrightarrow\P$. Let
$\widehat{\P}_{i}$ be the blow--up of $\P_i$ along
$(\tau^{i+1}\Delta)_i$ and $\widehat{X}_{i}$ be the strict
transform of $X_i$ in $\widehat{\P}_{i}$. Then\begin{itemize}
    \item $\widehat{X}_i\longrightarrow X$ \emph{is a smooth small
resolution satisfying the statement, for any $0\leq i\leq 5$}.
\end{itemize}
To prove this fact we have to check that:
\begin{itemize}
  \item[\emph{i.}] the exceptional locus of the resolution $\widehat{X}_i\longrightarrow X$ is actually composed by disjoint couples of $(-1,-1)$--curves intersecting in one point,
  \item[\emph{ii.}] the resolutions $\widehat{X}_i\longrightarrow X$ are to each other connected by flops of $(-1,-1)$--curves.
\end{itemize}
Let us prove \emph{i} locally, by explicitly computing the induced resolution of a singular point of type (\ref{cA2}) over the open subset $\mathcal{A}\subset\P$ defined in (\ref{apertoW}).
Up to an isomorphism we may always assume that
$B(1:0)\neq 0$, implying that $\Sing(X)\subset \mathcal{A}\cap X$. Let us
assume that $B(t_0:1)=0$, then
\[
    p_{t_0}:= ((0:0:1),(0:0:1),(t_0:1))\in \Sing(X)
\]
is a threefold cusp whose local equation (\ref{cA2}) can be
factored as follows
\begin{equation}\label{fatt.cusp}
    (X-U)(X+U)=(Y-V)(Y-\epsilon V)(Y-\epsilon^2 V)\ .
\end{equation}
Let $\mathcal{A}'$ be the section of $\mathcal{A}$ with the hyperplane $t=t_0$. Then
\begin{equation*}
  \mathcal{A}'\cap\tau^i\Delta = \left\{(X,Y,U,V,t_0)\in\C^5 |\ X -(-1)^i U =
  Y -\epsilon^i V = 0\right\}
  \cong \C^2\ .
\end{equation*}
Rewrite (\ref{fatt.cusp}) as $xy=uv[(1+\epsilon)v-\epsilon u]$, where
\begin{eqnarray*}
    &x=X -(-1)^i U\ ,\ y=X -(-1)^{i+1} U&\ ,\\\ &u=Y -\epsilon^i V\ ,\ v=Y -\epsilon^{i+1} V\ ,\ w=Y -\epsilon^{i+2} V&\ .
\end{eqnarray*}
The blow up $\P_i\rightarrow\P$ of $\tau^i\Delta$ induces over $\mathcal{A}'$ the blow up $\mathcal{A}'_i\rightarrow \mathcal{A}'$ of the plane $x=u=0$. The strict transform $X_i$ is then locally given by
\begin{equation*}
    \mathcal{A}'_i\cap X_i\ :\
    \left\{\begin{array}{c}
    \mu_1 x = \mu_0 u  \\
    \mu_0y = \mu_1v[(1+\epsilon)v-\epsilon u]  \\
  \end{array} \\
  \right.
\end{equation*}
where $\P^1[\mu_0,\mu_1]$ is the small exceptional locus of $\mathcal{A}'_i\cap X_i\rightarrow \mathcal{A}'\cap X$. Notice that $\mathcal{A}'_i\cap X_i$ is still singular admitting a node in the point $((\mathbf{0},t_0),(0:1))\in\mathcal{A}'_i\times\P^1$. On the other hand, the strict transform $(\tau^{i+1}\Delta)_i $ of $\tau^{i+1}\Delta$ is locally given by
\begin{equation}\label{trasformata1}
    \left\{\begin{array}{c}
    \mu_1 x = \mu_0 u  \\
    y = v =0  \\
  \end{array} \\
  \right.
\end{equation}
The blow up $\widehat{\P}_i\longrightarrow \P_i$ of $\P_i$ along $(\tau^{i+1}\Delta)_i$ induces over $\mathcal{A}'_i$ the blow up $\widehat{\mathcal{A}}'_i\rightarrow \mathcal{A}'_i$, along (\ref{trasformata1}).
Then the strict transform $\widehat{X}_i$ of $X$ is locally described
as the following codimension three closed subset of
$\mathcal{A}'\times\P^1[\mu]\times\P^1[\nu]$
\begin{equation*}
    \widehat{\mathcal{A}}'_i\cap\widehat{X}_i\ :\
    \left\{\begin{array}{c}
    \mu_1 x = \mu_0 u  \\
    \nu_1 y =\nu_0 v \\
    \mu_0\nu_0 = \mu_1\nu_1 [(1+\epsilon)v-\epsilon u]  \\
  \end{array} \\
  \right.\ .
\end{equation*}
Observe that:\begin{itemize}
    \item $\widehat{\mathcal{A}}'_i\cap\widehat{X}_i$ is smooth,
    \item $\widehat{\mathcal{A}}'_i\cap\widehat{X}_i\longrightarrow \mathcal{A}'\cap X$ is an
    isomorphism outside of $(\mathbf{0},t_0)\in \mathcal{A}'\cap X$, which locally
    represents $p_{t_0}\in\Sing(X)$,
    \item the exceptional fiber over $(\mathbf{0},t_0)\in \mathcal{A}'\cap X$ is
    described by the closed subset $\{\mu_0\nu_0 =
    0\}\subset\P^1[\mu]\times\P^1[\nu]$, which is precisely a
    couple of $\P^1$'s meeting in the point
    $((\mathbf{0},t_0),(0:1),(0:1))\in \widehat{\mathcal{A}}'_i\cap\widehat{X}_i\subset
    \mathcal{A}'\times\P^1\times\P^1$\ ,
    \item by construction any exceptional $\P^1$ is a $(-1,-1)$--curve.
\end{itemize}
To prove \emph{ii}, it suffices to show that:
\begin{itemize}
    \item for any $0\leq i\leq 5$ the following flops of
    $(-1,-1)$--curves exist:
    \[
    \xymatrix{X_i\ar@{<-->}[rr]\ar[rd]&&X_{i+2}\ar[ld]\\
                & X &}\quad,\quad
    \xymatrix{X_i\ar@{<-->}[rr]\ar[rd]&&X_{i+3}\ar[ld]\\
                & X &}\quad.
    \]
\end{itemize}
As before rewrite the local equation (\ref{fatt.cusp}) as
\[
x y = u v [(1+\epsilon)v-\epsilon u]\ \quad\text{in $\C[x,y,u,v]$}\ .
\]
Then locally $X_i$ corresponds to  blow up the plane $x=u=0$ of $\C^4$
while $X_{i+2}$ corresponds to blow up the plane $x=v=0$. Ignore
the term $[(1+\epsilon)v-\epsilon u]$: then our situation turns out to be similar to the
well known \emph{Koll\'{a}r quadric} (\cite{Kl} Example 3.2)
giving a flop
\[
    \xymatrix{X_i\ar@{<-->}[rr]\ar[rd]&&X_{i+2}\ar[ld]\\ & X &}.
\]
Analogously $X_{i+3}$ corresponds to blow up $y=u=0$ still getting
a flop
\[
    \xymatrix{X_i\ar@{<-->}[rr]\ar[rd]&&X_{i+3}\ar[ld]\\ & X &}.
\]
\end{proof}

\subsection{Deformations and resolutions}\label{def-ris}

Let $X=Y\times_{\P^1}Y$ be the Namikawa fiber product defined
above, starting from the bundle's homomorphism (\ref{W-Nami}). For
a general $b\in H^0(\mathcal{O}_{\P^1}(6))$, the singular locus
$\Sing(X)$ is composed by six cusps of type (\ref{cA2}). Let us
rewrite the local equation of this singularity as
follows
\begin{equation}\label{singolarità}
    x^2 - y^3 = z^2 - w^3\ .
\end{equation}
It is a singular point of Kodaira type $II\times II$. Moreover it
is a \emph{compound Du Val singularity} of $cA_2$ type i.e. a
threefold point $p$ such that, for a hyperplane section $H$
through $p$ (in the present case assigned e.g. by $w=0$), $p\in H$
is a Du Val surface singularity of type $A_2$ (see \cite{Reid80},
\S 0 and \S 2, and \cite{BPvdV84}, chapter III).

\noindent Recalling (\ref{ihs-T^1}), the
Kuranishi space of the cusp (\ref{singolarità}) is the
$\C$--vector space
\begin{equation}\label{T1}
    \T^1\cong T^1 \cong \mathcal{O}_{F,0}/J_F \cong
    \C\{x,y,z,w\}/\left((F)+ J_F\right)\cong \langle1,y,w,yw\rangle_{\C}
\end{equation}
where $F=x^2 - y^3 - z^2 + w^3$ and $J_F$ is the associated
Jacobian ideal. Then a miniversal deformation of (\ref{singolarità})
is given by the zero locus of
\begin{equation}\label{deformazione}
    F_{\Lambda}\ :\ x^2 - y^3 - z^2 + w^3 + \lambda + \mu y - \nu
    w + \sigma yw \in \C[x,y,z,w]\ ,\ \Lambda = (\lambda,\mu,-\nu, \sigma)\in \T^1\ .
\end{equation}
The fibre $\mathcal{X}_{\Lambda}=\{F_{\Lambda}=0\}$, of the miniversal deformation family $\mathcal{X}\to\T^1$, is singular if and
only if the Jacobian rank of the polynomial function $F_{\Lambda}$
is not maximum at some zero point of $F_{\Lambda}$. Singularities
are then given by $(0,y,0,w)\in \C^4$ such that
\begin{eqnarray}\label{condizioni}
  3y^2 -\sigma w - \mu &=& 0 \\
\nonumber
  3w^2 +\sigma y - \nu &=& 0 \\
\nonumber
  \sigma y w + 2\mu y - 2\nu w + 3\lambda &=& 0
\end{eqnarray}
where the first two conditions come from partial derivatives of
$F_{\Lambda}$ and the latter is obtained by applying the first two
conditions to the vanishing condition $F_{\Lambda}(0,y,0,w)=0$.

\begin{proposition}\label{3-singolarità}
A fibre of the miniversal deformation family $\mathcal{X}\to\T^1$ of the cusp
(\ref{singolarità}) admits at most three singular points.
\end{proposition}

\begin{proof} It is a direct consequence of conditions (\ref{condizioni}). Fix a point $$\Lambda=(\l,\mu,-\nu,\sigma)\in\T^1\ .$$
\noindent If $\sigma=0$ then conditions (\ref{condizioni}) become conditions (\ref{condizioni2}) below. The argument given in the proof of Proposition \ref{def-singolari} shows that common solutions of (\ref{condizioni2}) cannot be more than two.

\noindent Let us then assume that $\sigma \neq 0$. The first equation in (\ref{condizioni}) gives
\begin{equation}\label{eq1}
    \sigma w= 3y^2-\mu
\end{equation}
and the third equation multiplied by $\sigma$ gives
\begin{equation}\label{eq3}
    \sigma^2 y w + 2\mu\sigma y - 2\nu\sigma w + 3\lambda\sigma=0\ .
\end{equation}
Put (\ref{eq1}) in (\ref{eq3}) to get
\[
    3\sigma y^3 - 6\nu y^2 + \mu \sigma y + 2\mu \nu + 3
    \lambda\sigma=0\ .
\]
Fixing $\Lambda=(\l,\mu,-\nu,\sigma)\in\T^1$, the latter is a cubic polynomial in the unique unknown variable $y$. Then the common solutions of equations (\ref{condizioni}) cannot be more than 3.
\end{proof}

For any $p\in\Sing(X)$ let $U_p=\Spec\mathcal{O}_{F,p}$ be a representative of the complex space germ locally describing the singularity $p\in X$. By Theorem \ref{DGP teorema}, Definition
\ref{Def-definizione}, (\ref{T1}) and (\ref{deformazione})
\begin{equation}\label{kuranishi}
    \Def(U_p)\cong \T^1\cong\left\langle1,y,w,yw\right\rangle_{\C}
\end{equation}
Recalling the local Friedman diagram (\ref{Friedman-locale}), consider the localization map
$$\lambda_p:\Def(X)\cong\T^1_X\longrightarrow\Def(U_p)\cong \T^1\ .$$

\begin{proposition}\label{def-singolari}
The deformations of the cusp (\ref{singolarità}) induced by localizing a
deformation of the fiber product $X$ are based on a
hyperplane of the Kuranishi space $\T^1$ in (\ref{T1}). More precisely, using coordinates introduced in (\ref{deformazione}),
\[
    \forall p\in \Sing(X)\quad \im (\lambda_p)=S:=\{\sigma=0\}\subset \T^1\ .
\]
In particular every deformation parameterized by $S$ may admit at
most 2 singular points which are
\begin{itemize}
    \item ordinary double points if $\mu\cdot \nu \neq 0$,
    \item compound Du Val of type $cA_2$ if $\mu\cdot\nu = 0$ and precisely of Kodaira type
    \begin{eqnarray*}
      II\times II &\text{if}& \mu=\nu=0\ , \\
      I_1\times II && \text{otherwise.}
    \end{eqnarray*}
\end{itemize}
\end{proposition}

\begin{proof} Let $X=Y\times_{\P^1}Y$ be the Namikawa fiber product given by equations (\ref{Nami}). A deformation family $\mathcal{Y}$ of the rational elliptic surface $r:Y\rightarrow\P^1$ comes equipped with a natural morphism $\mathcal{R}:\mathcal{Y}\rightarrow\P^1$ such that $\mathcal{R}_{|Y}=r$.
Then a deformation family of $X$ is obtained as the fiber product $\mathcal{X}=\mathcal{Y}\times_{\P^1}\mathcal{Y}'$ of two deformation families of $Y$ (this fact is explained in several points of the Namikawa's paper \cite{N}, e.g. in the Introduction and in Remark 2.8). The Kuranishi space of the cusp $\{x^2 - y^3 = 0\}$ is given by
\[
    \T^1_{\text{cusp}} \cong \C\{x,y\}/(x^2-y^3,x,y^2)\cong \langle 1,y\rangle_{\C}
\]
meaning that a deformation of $Y$ can be locally given by the germ
\[
    x^2-y^3 + \lambda + \mu y - t^iB^{i}(t_0)= 0\quad,\quad (\lambda,\mu)\in
    \T^1_{\text{cusp}}\ ,
\]
where $t_0$ is a zero of $B(t)$ and $B^{i}(t_0)$ has the same meaning as in (\ref{eq.localeW2}).
Therefore, by applying an analogous local analysis as that given in the proof of Proposition \ref{prop:cuspidi} and recalling (\ref{deformazione}),  it turns out that a deformation of the threefold cusp (\ref{singolarità}) induced by localizing a deformation of $X$, is given by the germ of complex space
\begin{equation*}
    \left\{\begin{array}{c}
             x^2 - y^3 + \lambda_1 + \mu y - t^iB^{i}(t_0) =0\\
             z^2 - w^3 + \lambda_2 + \nu w - t^iB^{i}(t_0) =0
           \end{array}
    \right.
\end{equation*}
giving the following local equation
\[
    x^2 - y^3 + \lambda_1 + \mu y = z^2 - w^3 + \lambda_2 + \nu w\ .
\]
Then all such deformations span the subspace
\[
\{(\lambda_1 - \lambda_2,\mu,-\nu,0)\}=\{\sigma=0\}\subset \T^1\ .
\]
Setting $\sigma=0$ in conditions (\ref{condizioni}) gives the
following equations
\begin{eqnarray}\label{condizioni2}
  3y^2 - \mu &=& 0 \\
\nonumber
  3w^2 - \nu &=& 0 \\
\nonumber
  2\mu y - 2\nu w + 3\lambda &=& 0
\end{eqnarray}
which can be visualized in the $y,w$--plane as follows:
\begin{itemize}
    \item the first condition as two
    parallel and symmetric lines with respect to the $y$--axis; they may
    coincide with the $y$--axis when $\mu = 0$;
    \item the second condition as two
    parallel and symmetric lines with respect to the $w$--axis; they may
    coincide with the $w$--axis when $\nu = 0$;
    \item the last condition as a line in
    general position in the $y,w$--plane.
\end{itemize}
Clearly, fixing the point $\Lambda=(\l,\mu,-\nu,\sigma)\in\T^1$, it is not possible to have more than two distinct common
solutions of (\ref{condizioni2}) with respect to the variables $y,w$.

\noindent To analyze the singularity type, let
$p_{\Lambda}=(0,y_{\mu},0,w_{\nu})$ be a singular point of
\[
    F_{\Lambda}\ :\ x^2 - y^3 - z^2 + w^3 + \lambda + \mu y - \nu
    w =0
\]
and translate $p_{\Lambda}$ to the origin by replacing
\begin{equation*}
  y \longmapsto y+y_{\mu} \quad,\quad
  w \longmapsto w+w_{\nu}\ .
\end{equation*}
Then conditions (\ref{condizioni2}) impose that the translated
$F_{\Lambda}$ is
\[
    \widetilde{F}_{\Lambda} = x^2 - y^3 - z^2 + w^3 - 3y_{\mu}\ y^2 + 3w_{\nu}\ w^2
\]
giving the classification above.
\end{proof}

\begin{corollary}\label{3sings}
If the deformation $\mathcal{X}_{\Lambda}$ of the cusp (\ref{singolarità}), associated
with $\Lambda\in \T^1$, admits three distinct singular points then
$\Lambda\in \T^1\setminus S$, which is
\[
    \Lambda = (\lambda,\mu,-\nu,\sigma)\quad\text{with $\sigma\neq
    0$}\ .
\]
\end{corollary}

\begin{proposition}\label{3sings-effettive}
The locus of the Kuranishi space $\T^1$ in (\ref{T1}),
parameterizing deformations of the cusp (\ref{singolarità}) to 3
distinct nodes, is described by the plane smooth curve
    \begin{equation*}
    C = \left\{\sigma^3-27\lambda=\mu=\nu=0\right\}\subset
    \T^1
    \end{equation*}
transversally meeting the hyperplane $S$ in the origin of
$\T^1$. In particular, $(0,0,0,1)\in \T^1$ generates the tangent
space in the origin to the base of a $1^{\text{st}}$--order
deformation of the cusp $(\ref{singolarità})$ to three distinct nodes.
\end{proposition}

\begin{proof}
Consider conditions (\ref{condizioni}): since, by Corollary
\ref{3sings}, we can assume $\sigma\neq0$, the first equation
gives
\[
     w= \frac{3y^2-\mu}{\sigma}\ .
\]
Putting this in the second equation gives that
\[
    R_1:= 27y^4 -18\mu y^2+\sigma^3y+3\mu^2-\nu\sigma^2=0\ ,
\]
and in the third equation gives that
\[
    R_2:=3\sigma y^3 - 6\nu y^2 + \mu \sigma y + 2\mu \nu + 3
    \lambda\sigma=0\ .
\]
Therefore, fixing $\Lambda=(\l,\mu,-\nu,\sigma)\in\T^1$, consider $R_1$ and $R_2$ as polynomials in $\C[\l,\mu,\nu,\sigma][y]$, i.e. in the unique unknown variable $y$. Then  (\ref{condizioni}) admit three common solutions if and only
if $R_2$ divides $R_1$. Since the remainder of the division of
$R_1$ by $R_2$ in $\C[\l,\mu,\nu,\sigma][y]$ is
\[
    \frac{27}{\sigma^2}(4\nu^2-\mu\sigma^2)y^2 + \frac{1}{\sigma}
    (\sigma^4-36\mu\nu-27\lambda\sigma)y + \frac{1}{\sigma^2}(3\mu^2\sigma^2
    -\nu\sigma^4 - 36\mu\nu^2 - 54\lambda\nu\sigma)
\]
it turns out to be 0 if and only $\Lambda$ satisfies the following
conditions
\begin{eqnarray}\label{condizioni3}
  4\nu^2-\mu\sigma^2 &=& 0 \\
\nonumber
  \sigma^4-36\mu\nu-27\lambda\sigma &=& 0 \\
\nonumber
  3\mu^2\sigma^2
    -\nu\sigma^4 - 36\mu\nu^2 - 54\lambda\nu\sigma &=& 0
\end{eqnarray}
Then the first equation gives
\[
    \mu=4\frac{\nu^2}{\sigma^2} \ .
\]
Putting this in the second equations we get
\[
    \lambda=\frac{\sigma^3}{27}-\frac{16\nu^3}{3\sigma^3}\ .
\]
Then, from the third equation in (\ref{condizioni3}), we get the following factorization
\[
    \nu(4\nu-\sigma^2)(4\nu-\epsilon\sigma^2)(4\nu-\epsilon^2\sigma^2)=0
\]
where $\epsilon$ is a primitive cubic root of unity. All the
solutions of (\ref{condizioni3}) are then the following
\begin{eqnarray}\label{soluzioni}
    \Lambda_0=(\frac{1}{27}\sigma^3,0,0,\sigma)\ &,&\
    \Lambda_1=(-\frac{5}{108}\sigma^3,\frac{1}{4}\sigma^2,-\frac{1}{4}\sigma^2,\sigma)\ ,\\
\nonumber
    \Lambda_2=(-\frac{5}{108}\sigma^3,\frac{\epsilon^2}{4}\sigma^2,-\frac{\epsilon}{4}\sigma^2,\sigma)\
    &,&\ \Lambda_3=(-\frac{5}{108}\sigma^3,\frac{\epsilon}{4}\sigma^2,-\frac{\epsilon^2}{4}\sigma^2,\sigma)\
    .
\end{eqnarray}
Let us first consider the second solution $\Lambda_1$. In this
particular case $R_2$ becomes
\[
    R_2=3\sigma y^3- \frac{3}{2}\sigma^2 y^2 + \frac{1}{4}\sigma^3
    y -\frac{1}{72}\sigma^4 =
    3\sigma\left(y-\frac{\sigma}{6}\right)^3\ ,
\]
meaning that \emph{$\Lambda_1$ is actually the base of a trivial
deformation of (\ref{singolarità})} since the fiber
associated with $\sigma$ admits the unique singular point
$(0,\sigma/6,0,-\sigma/6)$ which is still a cusp of type
(\ref{singolarità}). Moreover solutions $\Lambda_2$ and
$\Lambda_3$ give trivial deformations too, since they can be
obtained from $\Lambda_1$ by replacing
\begin{eqnarray*}
  &\text{either}& \xymatrix@1{y\ \ar@{|->}[r]&\ \epsilon y}\ ,\
  \xymatrix@1{w\ \ar@{|->}[r]&\ \epsilon^2 w}\quad\text{(giving $\Lambda_2$)} \\
  &\text{or}& \xymatrix@1{y\ \ar@{|->}[r]&\ \epsilon^2 y}\ ,\
  \xymatrix@1{w\ \ar@{|->}[r]&\ \epsilon w}\quad\text{(giving $\Lambda_3$)\ .}
\end{eqnarray*}
It remains to consider the first solution $\Lambda_0$. In this
case $R_2$ becomes
\[
    R_2= y^3+ \frac{\sigma^3}{27}=\left(y+\frac{\sigma}{3}\right)
    \left(y+\frac{\epsilon\sigma}{3}\right)
    \left(y+\frac{\epsilon^2\sigma}{3}\right) ,
\]
then the deformation $X_{\sigma},\ \sigma\neq0$, turns out to
admit three distinct nodes given by
\begin{equation}\label{3nodi}
    \left(0,-\frac{\sigma}{3},0,\frac{\sigma}{3}\right)\ ,\
    \left(0,-\frac{\epsilon\sigma}{3},0,\frac{\epsilon^2\sigma}{3}\right)\ ,\
    \left(0,-\frac{\epsilon^2\sigma}{3},0,\frac{\epsilon\sigma}{3}\right)\ .
\end{equation}
Notice that the base curve $\Lambda_0\subset \T^1\cong \C^4$ is
actually the plane smooth curve
$C=\{\sigma^3-27\lambda=\mu=\nu=0\}$ meeting the hyperplane $S=\{\sigma=0\}$ only in the
origin, where they are transversal since a
tangent vector to $C$ in the origin is a multiple of $(0,0,0,1)$.
The statement is then proved by thinking $\T^1$ as the tangent
space in the origin to the germ of complex space $\Def({U_0})$ and
representing the functor of $1^{\text{st}}$--order deformation of
the cusp (\ref{singolarità}).
\end{proof}

Let $\widehat{X}\stackrel{\phi}{\longrightarrow} X$ be one of the
six small resolutions constructed in Proposition
\ref{nami-risoluzione} and consider the \emph{localization near to
$p\in\Sing (X)$}
\begin{equation}\label{localizzazione}
    \xymatrix{\widehat{U}_p:=\phi^{-1}(U_p)\ar@{^{(}->}[r]\ar[d]^-{\varphi}
    & \widehat{X}\ar[d]^-{\phi}\\
    U_p:=\Spec\mathcal{O}_{F,p}\ar@{^{(}->}[r]&X}
\end{equation}
and consider the associated local Friedman diagram (\ref{Friedman-locale}), which is the following commutative diagram between Kuranishi
spaces
\begin{equation}\label{Def-localizzazione}
    \xymatrix{\Def(\widehat{X})\cong H^1(\Theta_{\widehat{X}})\ar[r]^-{\lambda_{E_p}}
    \ar@{^{(}->}[d]^-{\delta_1}&\Def(\widehat{U}_p)\cong H^0(R^1\phi_*\Theta_{\widehat{U}_p})\ar@{^{(}->}[d]^-{\delta_{loc,p}}\\
    \Def(X)\cong\T^1_X\ar[r]^-{\lambda_p}&\Def(U_p)\cong \T^1}
\end{equation}

\begin{theorem}\label{immagine0}
The image of the map $\delta_{loc,p}$ in diagram
(\ref{Def-localizzazione}) is the plane smooth curve
$C\subset T^1$ defined in Proposition \ref{3sings-effettive}. In
particular this means that
\begin{enumerate}
   \item[(a)]\ $\dimdef(\widehat{U}_p)=1$\ ,
   \item[(b)]\ $\im(\lambda_p)\cap\im(\delta_{loc,p})= 0$\ ,
   \item[(c)]\ $\im(\lambda_{E_p}) = 0$\ .
\end{enumerate}
\end{theorem}

\begin{proof} By the construction of the resolution
$\widehat{X}\stackrel{\phi}{\longrightarrow} X$, $\im(\delta_{loc,p})\subset\T^1$ parameterizes all the deformations of $U_p$ induced by deformations of $\widehat{U}_p$. Then a fiber of the versal deformation family of $U_p$ over a point in $\im(\delta_{loc,p})\subset\T^1$ has local equation respecting the factorization (\ref{fatt.cusp}) of the local equation (\ref{cA2}). A general deformation respecting such a
factorization can be written as follows
\begin{equation}\label{cusp-def}
    (X-U+\xi)(X+U+\upsilon)=(Y-V+\alpha)(Y-\epsilon V+\beta)(Y-\epsilon^2 V+\gamma)
\end{equation}
for $(\xi,\upsilon,\alpha,\beta,\gamma)\in \C^5$. By means of the
translation
$$X\longmapsto X-\frac{\xi+\upsilon}{2}\quad,\quad U\longmapsto U+\frac{\xi-\upsilon}{2}$$
the left part of (\ref{cusp-def}) becomes $X^2-U^2$. After some calculation on the right part, (\ref{cusp-def}) can then be rewritten as follows
\begin{eqnarray}\label{Fa}
    F_{\mathbf{a}}&:=& F -
    (\alpha+\beta+\gamma)Y^2
    -(\alpha+\epsilon^2\beta+\epsilon\gamma)V^2 -
    (\alpha+\epsilon\beta+\epsilon^2\gamma)YV \\
\nonumber
    &&-(\alpha\gamma+\alpha\beta+\beta\gamma)Y +
    (\beta\gamma+\epsilon\alpha\gamma+\epsilon^2\alpha\beta)V -
    \alpha\beta\gamma=0
\end{eqnarray}
where $F:=X^2 - U^2 -Y^3 + V^3$ and
$\mathbf{a}:=(\alpha,\beta,\gamma)\in A\cong\C^3$. Consider the
 deformation
$f:\mathcal{U}\longrightarrow (A,0)$ defined by setting
$$\forall\ \mathbf{a}\in A\quad\mathcal{U}_{\mathbf{a}}:=f^{-1}(\mathbf{a})=\{F_{\mathbf{a}}=0\}\ .$$
The following facts
then occur:
\begin{enumerate}
    \item \emph{the fibre $\mathcal{U}_{\mathbf{a}}$
    is isomorphic to the central fibre $\mathcal{U}_0$ if and only if
    $\mathbf{a}$ is a point of the plane $\pi=\{\alpha+\epsilon\beta+\epsilon^2\gamma=0\}\subset A$;
    in particular $\mathcal{U}|_{\pi}\longrightarrow \pi$ is a trivial deformation;}
    \item \emph{the open subset $V:=A\setminus\pi$ is the base
    of a deformation of the cusp $\mathcal{U}_0=\{F=0\}$ to 3
    distinct nodes;}
    \item \emph{there exists a morphism of germs of complex spaces $g:(A,0)\longrightarrow (\T^1,0)$ to the Kuranishi space
    $\T^1$ described in (\ref{T1}), such
    that $\im g$ turns out to be precisely the plane smooth
    curve $C$ defined in Proposition \ref{3sings-effettive} and
    parameterizing the deformation of the cusp
    (\ref{singolarità}) to three distinct nodes.}
\end{enumerate}
Let us postpone the proof of these facts to observe that fact (3)
means that the deformation
$\mathcal{U}\to A$ is the
\emph{pull--back by $g$} of the versal deformation
$\mathcal{V}\to \T^1$ i.e.
$\mathcal{U}=A\times_{\T^1}\mathcal{V}$. Then $C=\im(\delta_{loc,p})$
proving the first part of the statement. Part (a) of the statement then follows by
recalling that $\delta_{loc,p}$ is injective. Moreover Propositions
\ref{def-singolari} and \ref{3sings-effettive} allow one to prove
part (b). Finally part (c) follows by (b), the injectivity of
$\delta_{loc,p}$ and the commutativity of diagram
(\ref{Def-localizzazione}).

\noindent Let us then prove facts (1), (2) and (3) stated above.

\halfline (1)\ ,\ (2)\ :\quad this facts are obtained analyzing
the common solutions of
\[
    F_{\mathbf{a}}= \partial_X F_{\mathbf{a}}= \partial_Y F_{\mathbf{a}}=
    \partial_U F_{\mathbf{a}}=\partial_V F_{\mathbf{a}}=0\ .
\]
Since $\partial_X F_{\mathbf{a}} = 2X$ and $\partial_U
F_{\mathbf{a}} = 2U$, we can immediately reduce to look for the
common solutions $(0,Y,0,V)$ of
\begin{equation}\label{sistema}
    F_{\mathbf{a}}(0,Y,0,V)= \partial_Y F_{\mathbf{a}} =
    \partial_V F_{\mathbf{a}}=0\ .
\end{equation}
After some calculations one finds that these common solutions are given by
\begin{eqnarray}\label{3nodi-def}
  p_1 = &\left(0,\frac{\varepsilon\beta-\gamma}
  {1-\epsilon},0,\frac{\beta -\gamma}{\varepsilon (1-\varepsilon)}\right)& \\
\nonumber
  p_2 = &\left(0,\frac{\varepsilon^2\alpha-\gamma}
  {1-\epsilon^2},0,\frac{\alpha -\gamma}{1-\varepsilon^2}\right)&  \\
\nonumber
  p_3 = &\left(0,\frac{\varepsilon\alpha-\beta}
  {1-\epsilon},0,\frac{\alpha -\beta}{1-\varepsilon}\right)&
\end{eqnarray}
which have to be necessarily distinct since
\begin{eqnarray*}
  \frac{\beta -\gamma}{\varepsilon (1-\varepsilon)}=\frac{\alpha -\gamma}{1-\varepsilon^2} &\Longleftrightarrow&
  \frac{\alpha -\gamma}{1-\varepsilon^2}= \frac{\alpha -\beta}{1-\varepsilon}\ \Longleftrightarrow\
  \frac{\alpha -\beta}{1-\varepsilon}=\frac{\beta -\gamma}{\varepsilon (1-\varepsilon)}\\
   &\Longleftrightarrow& \alpha+\epsilon\beta+\epsilon^2\gamma=0\
   .
\end{eqnarray*}
On the other hand, if $\alpha+\epsilon\beta+\epsilon^2\gamma=0$
then we get the unique singular point $p_1=p_2=p_3$
which is still a threefold cusp.

\halfline (3)\ :\quad Look at the definition (\ref{Fa}) of
$F_{\mathbf{a}}$ and construct $g$ as a composition $g=i\circ p$
where
\begin{itemize}
    \item \emph{$p:(A,0)\cong(\C^3,0)\longrightarrow (A,0)\cong(\C^3,0)$ is a linear map of rank 1 whose
kernel is the plane $\pi\subset A$ defined in (1)},
    \item \emph{$i:(A,0)\cong(\C^3,0)\to(\T^1,0)\cong(\C^4,0)$ is the map
    $(\alpha,\beta,\gamma)\mapsto(\lambda,\mu,-\nu,\sigma)$ given by
\begin{eqnarray*}
    \lambda=-\alpha\beta\gamma\ &,&\
    \mu=-\alpha\gamma-\alpha\beta-\beta\gamma\ ,\\
    \nu=-\beta\gamma-\epsilon\alpha\gamma-\epsilon^2\alpha\beta\
    &,&\ \sigma=-\alpha-\epsilon\beta-\epsilon^2\gamma\ ;
\end{eqnarray*}
    then, by (2) and Proposition \ref{3sings-effettive}, necessarily
    $\im i =C$ and $i|_{\im p}$ is the rational
    parameterization $\Lambda_0$ given in (\ref{soluzioni})}.
\end{itemize}
The linear map $p$ has to annihilate the coefficients of $Y^2$ and
$V^2$ in (\ref{Fa}) i.e.
\[
    \alpha+\beta+\gamma=\alpha+\epsilon^2\beta+\epsilon\gamma=0\ .
\]
Then we get the following conditions
\[
\im p
=\left\langle(\epsilon,1,\epsilon^2)\right\rangle_{\C}\subset A\
, \ \ker
p=\pi=\left\langle(-\epsilon,1,0),(-\epsilon^2,0,1)\right\rangle_{\C}\subset A
\]
which determine $p$, up to a multiplicative constant $k\in\C$, as
the linear map
represented by the rank 1 matrix\quad $k\cdot\left(%
\begin{array}{ccc}
  1 & \epsilon & \epsilon^2 \\
  \epsilon^2 & 1 & \epsilon \\
  \epsilon & \epsilon^2 & 1 \\
\end{array}%
\right)$.\quad Then
\[
    p(\mathbf{a})=k(\epsilon^2\alpha+\beta+\epsilon\gamma)
    \cdot\left(\epsilon,1,\epsilon^2\right)
\]
and
\[
    g(\mathbf{a})=i\circ p\ (\mathbf{a}) =
    \left(-k^3(\alpha+\epsilon\beta+\epsilon^2\gamma)^3
    ,0,0,-3k(\alpha+\epsilon\beta+\epsilon^2\gamma)\right)
\]
which satisfies equations $\sigma^3-27\lambda=\mu=\nu=0$ of
$C\subset \T^1$.
\end{proof}

\begin{remark}\label{rem:nodef} Propositions \ref{def-singolari} and
\ref{3sings-effettive} and Theorem \ref{immagine0} give a detailed
and revised version of what observed by Y.Namikawa in \cite{N}
Examples 1.10 and 1.11 and Remark 2.8. In fact point (c) of
Theorem \ref{immagine0} proves the following
\end{remark}

\begin{theorem}\label{thm:nodef} In the notation introduced above, every global deformation of
the small resolution $\widehat{X}$ induces only trivial local
deformations of a neighborhood of the exceptional fibre
$\phi^{-1}(p)$ over a cusp $p\in\Sing(X)$.
\end{theorem}

\section{A small and non-simple geometric transition}\label{s:gt}
Finally we propose an example of a small geometric transition which is not a simple gt i.e. it is not a \emph{deformation of a conifold transition}, as defined in the following. This example has been already sketched in \S 9.2 of \cite{Rdef}. Thanks to the detailed analysis of the Kuranishi space of a Namikawa fiber product presented above, all the following statements are now completely proved.

Let us first of all recall the main definitions.

\begin{definition}[see \cite{R1} and references therein]\label{gtdef}
Let $\widehat{X}$ be a \cy threefold and $\phi :
\widehat{X}\longrightarrow X$ be a \emph{birational contraction}
onto a \emph{normal} variety. Assume that there exists a \emph{\cy
smoothing} $\widetilde{X}$ of $X$. Then the process of going from
$\widehat{X}$ to $\widetilde{X}$ is called a \emph{geometric
transition} (for short \emph{transition} or \emph{gt}) and
denoted either by $T(\widehat{X},X,\widetilde{X})$ or by the diagram
\begin{equation}\label{small gt}
\xymatrix@1{\widehat{X}\ar@/_1pc/ @{.>}[rr]_{T}\ar[r]^{\phi}&
                X\ar@{<~>}[r]&\widetilde{X}}
\end{equation}
A gt $T(\widehat{X},X,\widetilde{X})$ is called \emph{conifold}
if $X$ admits only \emph{ordinary double points} (nodes or
o.d.p.'s) as singularities. Moreover a gt $T(\widehat{X},X,\widetilde{X})$
will be called \emph{small} if $\codim_{\widehat{X}} \Exc (\phi)
>1$, where $\Exc (\phi)$ denotes the exceptional locus of $\phi$.
\end{definition}

The most well known example of a gt is given by a generic
quintic threefold $X\subset\P^4$ containing a plane. One can check
that $\Sing (X)$ is composed by 16 nodes.
Look at the strict transform of $X$, in the blow--up of $\P^4$
along the contained plane, to get the resolution $\widehat{X}$,
while a generic quintic threefold in $\P^4$ gives the smoothing
$\widetilde{X}$. Due to the particular nature of $\Sing (X)$ the
gt $T(\widehat{X},X,\widetilde{X})$ is actually an example of a
conifold transition and hence of a small gt, too.

\begin{remark}
Let $\gt$ be a small gt. Then $\Sing(X)$ is composed at most by
terminal singularities of index 1 which turns out to be
isolated hypersurface singularities (actually of
\emph{compound Du Val type} \cite{Reid80}, \cite{Reid83}). The exceptional locus $\Exc(\phi: \widehat{X}\rightarrow X)$ is then composed by a finite number of trees of transversally intersecting rational curves, dually represented by ADE Dynkin diagrams \cite{Laufer81}, \cite{Pinkham81},
\cite{Morrison85}, \cite{Friedman86}.

\noindent Moreover
\cite[Thm.~A]{Namikawa94} allows us to conclude
that $\Def(X)$ is smooth.
Therefore
\begin{equation}\label{kuranishi numero bis}
    \dimdef(X)=\dim_{\C}\T^1_X\ .
\end{equation}
Moreover the Leray spectral sequence of the sheaf
$\phi_*\Theta_{\widehat{X}}$ gives
\[
    h^0(\Theta_X) = h^0(\phi_*\Theta_{\widehat{X}}) = h^0(\Theta_{\widehat{X}})
    = h^{2,0}(\widehat{X}) = 0
\]
where the first equality on the left is a consequence of the
isomorphism $\Theta_X\cong\phi_*\Theta_{\widehat{X}}$ (see
\cite{Friedman86} Lemma (3.1)) and the last equality on the right
is due to the \cy condition for $\widehat{X}$. Then Theorem
\ref{universality thm} and (\ref{T0}) allow to conclude that $X$
\emph{admits a unique smoothing component giving rise to a universal effective family of \cy deformations}.
\end{remark}

In the case of small geometric transitions we may then set the following

\begin{definition}\label{def:def}
Two small geometric transitions
$$T_1(\widehat{X}_1,X_1,\widetilde{X}_1)\ ,\ T_2(\widehat{X}_2,X_2,\widetilde{X}_2)$$
are \emph{direct deformation equivalent} (i.e. \emph{direct def-equivalent}) if the associated birational morphisms $\phi_i: \widehat{X}_i\rightarrow X_i\ ,\ i=1,2$\,, are deformations of each other as defined in \S\,\ref{ssez:morf-def}.

\noindent The equivalence relation of small geometric transitions generated by direct def-equivalence is called \emph{def-equivalence} (or \emph{deformation type}) of small geometric transitions.

\noindent In particular a small gt $T(\widehat{X},X,\widetilde{X})$ is called \emph{simple} if it is def-equivalent to a conifold transition.
\end{definition}

Actually def-equivalence of geometric transitions is a more complicated concept which reduces to the given Definition \ref{def:def} in the case of small geometric transitions. The interested reader is referred to \cite{Rdef} \S 8.1 for any further detail.

A direct consequence of Theorem \ref{immagine0} is then the following

\begin{corollary}\label{cor:nodef} There exist small geometric transitions which are not def-e\-qui\-va\-lent to a conifold transition.
\end{corollary}

\begin{proof}
Consider the Namikawa fiber product $X:=Y\times_{\P^1}Y$, a smooth resolution $\phi:\widehat{X}\rightarrow X$, as defined in Proposition \ref{nami-risoluzione} and a fiber product of generic rational elliptic surfaces $\widetilde{X}:=Y'\times_{\P^1}Y''$. Then $T(\widehat{X},X,\widetilde{X})$ is a small gt which can't be direct def-equivalent to a conifold transition for what observed in Remark \ref{rem:nodef}.

\noindent Assume now that $T(\widehat{X},X,\widetilde{X})$ is def-equivalent to a conifold transition. This means that there exists a finite chain of morphism deformations connecting the resolution  $\phi:\widehat{X}\rightarrow X$ with a conifold resolution. Hence there exists at least one of those morphism deformations locally inducing a non-trivial deformation of the exceptional locus, so violating condition (c) of Theorem \ref{immagine0}.
\end{proof}

\begin{remark}[The Friedman diagram of the Namikawa fiber product]\label{rem:Fr-vs-Nami} As a final result let us give a complete account of the Friedman diagram (\ref{Friedman-diagramma}) in the case of the small resolution $\phi:\widehat{X}\to X$ of a cuspidal Namikawa fiber product $X=Y\times_{\P^1}Y$.

\noindent Let us start by computing the Kuranishi number $\dimdef(X)$, by an easy moduli computation. In fact the moduli of the elliptic surface $Y$ are 8 since they are given by the moduli of an elliptic pencil in
$\P^2$. Moreover the 6 cuspidal fibers are parameterized by the
roots in $\P^1$ of a general element in
$H^0(\mathcal{O}_{\P^1}(6))$ up to the action of the projective
group $\P\GL(2)$. Then
\[
    \dimdef(X)=2\cdot\dimdef(Y)+h^0(\mathcal{O}_{\P^1}(6))-\dim
    \GL(2,\C)= 2\cdot 8 + 7 - 4 = 19\ .
\]
Recalling the existence of the small geometric transition $T(\widehat{X},X,\widetilde{X})$, described in the proof of Corollary \ref{cor:nodef}, it turns out that
\begin{equation*}
    \dimdef(X)=\dimdef(\widetilde{X})=h^{1,2}(\widetilde{X}) =19 =h^{1,1}(\widetilde{X})
\end{equation*}
since $\widetilde{X}$ is a \cy 3--fold with $\chi(X)=0$ (\cite{S} \S2). Recalling that
\begin{itemize}
  \item the small resolution $\phi:\widehat{X}\to X$ is obtained as a composition of 2 blow-ups,
  \item the exceptional locus $\Exc(\phi)$ has 12 irreducible components,
  \item every cusp of $X$ has Milnor (and Tyurina) number 4,
\end{itemize}
then \cy conditions on $\widehat{X}$ and Theorem 7 in \cite{hom-type} gives that
\begin{eqnarray*}
\dimdef(\widehat{X}) = h^{1}(\Theta_{\widehat{X}})= h^{1,2}(\widehat{X}) &=& h^{1,2}(\widetilde{X}) - 16 = 3\\
\dim\T^2_{\widehat{X}} = h^{2}(\Theta_{\widehat{X}}) = h^{1,1}(\widehat{X})&=& h^{1,1}(\widetilde{X}) +2 = 21
\end{eqnarray*}
Moreover recall that:
\begin{itemize}
  \item by Proposition \ref{prop:localizzazione-diff}(1) and Theorem \ref{immagine0}(a), $h^0(R^1\phi_*\Theta_{\widehat{X}})=6$,
  \item by Theorem \ref{immagine0}(c) the localization map $\l_E$ in (\ref{Friedman-diagramma}) is trivial,
  \item by Proposition \ref{prop:localizzazione-diff}(2) and relation (\ref{T1}), $\dim T^1_X=24$,
  \item the last horizontal maps on the right in the Friedman diagram (\ref{Friedman-diagramma}) are, in this case, surjective since both related with spectral sequences having $E^{1,1}_{\infty}=E^{0,2}_{\infty}=0$.
\end{itemize}
Putting all together, the Friedman diagram (\ref{Friedman-diagramma}) becomes the following one
\begin{equation*}
    \xymatrix{0\ar[r]&\C^3\ar[r]^-{\cong}\ar@{=}[dd]&\C^3
    \ar[rr]^-{\lambda_E=0}\ar@{^{(}->}[dd]^{\delta_1}\ar[rd] &&
    \C^6\ar@{^{(}->}[r]^-{d_2}\ar@{^{(}->}[dd]^-{\delta_{loc}}&
    \C^{27}\ar@{>>}[r]\ar@{=}[dd] & \C^{21}\ar@{>>}[dd]^-{\delta_2}\ar[r]&0\\
    &&&0\ar[ru]&&&&\\
            0\ar[r]&\C^3\ar@{^{(}->}[r] & \C^{19}\ar[rr]^-{{\lambda_P}} &&
              \C^{24}\ar[r]
              &\C^{27}\ar@{>>}[r]&\C^{19}\ar[r]&0}
\end{equation*}
In particular notice that
$$h^1(X,\Theta_X)=h^1(\widehat{X},\Theta_{\widehat{X}})=3$$ contradicting the sufficient condition (13) in Proposition 32 of \cite{Rdef} which has to be satisfied by a simple gt. This is consistent with the previous Corollary \ref{cor:nodef}.
\end{remark}

\halfline
\noindent\textbf{Acknowledgements}
A first draft of the present paper was written on a
 visit to the Dipartimento di Matematica e Applicazioni
of the Universit\`{a} di Milano Bicocca and the Department of
Mathematics of the
University of Pennsylvania. I would like to thank the Faculties of both
Departments for warm hospitality and in particular F.~Magri,
R.~Paoletti and S.~Terracini from the first Department and
R.~Donagi, A.~Grassi, T.~Pantev and J.~Shaneson from the second
Department. Special thanks are due to A.~Grassi for beautiful and
stimulating conversations. I am also greatly indebted to
B.~van~Geemen for useful suggestions.



\end{document}